\documentclass{amsart}
\usepackage{amsmath}
\usepackage{amsfonts}
\usepackage{eucal}
\usepackage{amsthm}
\numberwithin{equation}{section}
\usepackage{comment}

\newcommand{\jap}[1]{\langle #1 \rangle}

\def\a{\alpha}
\def\b{\beta}
\def\c{\gamma}
\def\d{\delta}
\def\e{\varepsilon}
\def\f{\varphi}
\def\g{\psi}

\def\l{\lambda}
\def\m{\mu}

\def\s{\sigma}

\def\x{\xi}
\def\y{\eta}

\renewcommand{\O}{\Omega}

\newcommand{\Op}{\mathrm{Op}}

\def\re{\mathbb{R}}

\def\pa{\partial}

\renewcommand{\Re}{\text{{\rm Re}\;}}
\renewcommand{\Im}{\text{{\rm Im}\;}}

\newcommand{\supp}{\text{{\rm supp}\;}}

\newcommand{\Ker}{\text{{\rm Ker}\;}}
\newcommand{\Ran}{\text{{\rm Ran}\;}}

\newcommand{\sgn}{\text{{\rm sgn}\;}}

\newtheorem{thm}{Theorem}[section]
\newtheorem{lem}[thm]{Lemma}
\newtheorem{prop}[thm]{Proposition}
\newtheorem{cor}[thm]{Corollary}
\newtheorem{defn}[thm]{Definition}

\theoremstyle{definition}
\newtheorem{ass}{Assumption}

\theoremstyle{remark}
\newtheorem{rem}[thm]{Remark}

\title[]%
{Scattering theory for repulsive Schr\"odinger operators and applications to limit circle problem}
\author{Kouichi Taira}
\begin{document}
\maketitle

\begin{abstract}
In this note, we study existence of the outgoing/incoming resolvents of repulsive Schr\"odinger operators which may not be essentially self-adjoint on the Schwartz space. As a consequence, we construct $L^2$-eigenfunctions associated with complex eigenvalues by a standard technique of scattering theory. In particular, we give another proof of the classical result via microlocal analysis: The repulsive Schr\"odinger operators with large repulsive exponent are not essentially self-adjoint on the Schwartz space.
\end{abstract}

\section{Introduction}
In this paper, we consider the following repulsive Schr\"odinger operator on $\re^n$:
\begin{align}\label{rep}
P=P_{\a}=-\Delta-\jap{x}^{2\a}+\Op(V),\quad \a>1,
\end{align}
where $\jap{x}=(1+|x|^2)^{1/2}$ and $\Op(V)$ is the Weyl quantization of a symbol $V:\re^{2n}\to \re$. We set 
\begin{align*}
P_0=P_{0,\a}=-\Delta-\jap{x}^{2\a}, \quad \a>1.
\end{align*}
Let $p_0(x,\x)=|\x|^2-\jap{x}^{2\a}$ and $p(x,\x)=p_0(x,\x)+V(x,\x)$. In this paper, we always assume the following assumptions. 
\begin{ass}\label{assa}
We set
Suppose that $V$ is of the form
\begin{align*}
V(x,\x)=\sum_{j,k=1}^na_{jk}(x)\x_j\x_k+\sum_{j=1}^nb_j(x)\x_j+c(x),
\end{align*}
where $a_{jk}=a_{kj}$, $b_j$ and $c$ are real-valued smooth functions on $\re^n$ and satisfy 
\begin{align*}
&|\pa_{x}^{\b}a_{jk}(x)|\leq C_{\b}(1+|x|)^{-\m-|\b|},\, |\pa_{x}^{\b}b_{j}(x)|\leq C_{\b}(1+|x|)^{\a-\m-|\b|},\\
&|\pa_{x}^{\b}c(x)|\leq C_{\b}(1+|x|)^{2\a-\m-|\b|}.
\end{align*}
with some $0<\m<1/2$ and $C_{\b}>0$.
\end{ass}
In particular, $\Op(V)$ is a symmetric differential operator and
\begin{align*}
V(x,\x)=\sum_{j=0}^2V_j(x,\x),\quad V_j\in S^{j, \a (2-j)-\m}.
\end{align*}

\begin{ass}\label{assb}
For any $M>0$
\begin{align*}
|p(x,\x)|\geq C\jap{\x}^2,\quad |x|\leq M, |\x|\geq R_0
\end{align*}
with some $C>0$ and $R_0>0$.
\end{ass}

We study stationary scattering theory of $P$ and give an application to limit circle problem. The usual scattering theory is based on the limiting absorption principle: the resolvent bound
\begin{align}\label{resbd}
\sup_{\re z\in I,\, \Im z\neq 0}\|\jap{x}^{-1/2-0}(-\Delta+V-z)^{-1} \jap{x}^{-1/2-0}\|_{L^2\to L^2}<\infty
\end{align}
and existence of the boundary values of the resolvent
\begin{align}\label{resbovalue}
\lim_{\pm\Im z\to 0}\jap{x}^{-1/2-0}(-\Delta+V-z)^{-1} \jap{x}^{-1/2-0}.
\end{align}
$(\ref{resbd})$ is used in order to prove existence and completeness of the wave operators. $(\ref{resbovalue})$ is used for a construction of generalized eigenfunctions of the stationary Schr\"odinger equation:
\begin{align*}
(-\Delta+V-z)u=0.
\end{align*}
The difficulty in the case of $P$ with $\a>1$ lies in the lack of essential self-adjointness of $P$ on $\mathcal{S}(\re^n)$. Since $P$ may have many self-adjoint extensions, "the boundary value of the resolvent" seems meaningless. The recent progress in the microlocal analysis gives another definition of the outgoing/incoming resolvents of pseudodifferential operators under some dynamical conditions. See \cite{FS} for the Anosov vector fields, \cite{BVW} and \cite{V2} for the d'Alembertians in the scattering Lorentzian spaces. We apply this technique to the repulsive Schr\"odinger operator $P$ even for $\a>1$ and  prove existence of the outgoing/incoming resolvents. Moreover, we show that $P$ has many eigenfunctions associated with the eigenvalues $\l\in \mathbb{C}$ except for the discrete set.  As a corollary, we give another proof of that $P$ is not essentially self-adjoint for $\a>1$ in view of scattering and microlocal theory. This is a classical result which is known as a typical limit circle case (for example, see \cite{RS}) when $\Op(V)$ is a multiplication operator.  It seems to be new result when $\Op(V)$ is not a multiplication operator. 

The repulsive Schr\"odinger operator is studied by several authors when $\Op(V)$ is a multiplication operator. Time-dependent scattering theory of the operator $(\ref{rep})$ for $0<\a\leq 1$ is studied in \cite{BCHM} in the short-range case. They prove existence and completeness of the wave operator and existence of the asymptotic velocity. They also study that existence of the outgoing/incoming resolvent and the absence of $L^2$-eigenvalues. The recent works in \cite{It1} and \cite{It2} extend some results in \cite{BCHM} for the long-range case. Moreover, in  \cite{It1}, the author of these papers proves the absence of eigenvalues in the Besov space is proved, where the order of Besov space is $\frac{\a-1}{2}$. This result is an extension of well-known results for the usual Schr\"odinger operators ($\a=0$) to the repulsive Schr\"odinger operators ($0<\a\leq 1$).

From the usual stationary scattering theory of $-\Delta$, we know that:
\begin{itemize}
\item Eigenfunctions of $-\Delta$ associated with positive eigenvalues do not exist in the threshold weighted $L^2$-space: $L^{2,-\frac{1}{2}}$.
\item There are many eigenfunctions right above $L^{2,-\frac{1}{2}}$:
\begin{align*}
-\Delta u=\l u,\quad u\in \bigcap_{s>1/2}L^{2,-s}
\end{align*}
for each $\l>0$.
\end{itemize}
The result in \cite{It1} and \cite{It2} suggests that the above results hold for the repulsive Schr\"odinger operator with $0<\a\leq 1$ with threshold weight $\frac{\a-1}{2}$. It is expected that  these results also hold for $\a>1$. In this paper, we almost justify these and we prove the existence of non-trivial $L^2$-solution to 
\begin{align*}
(P-z)u=0
\end{align*}
for $z\in \mathbb{C}$ except for a discrete subset of $\mathbb{C}$.

We introduce the variable order weighted $L^2$-space $L^{2,k+tm(x,\x)}$, where $k, t\in \re$ and $m$ is a real-valued function on the phase space $\re^{2n}$. Though we give a precise definition of $L^{2,k+tm(x,\x)}$ in Appendix \ref{appB}, we state properties of $L^{2,k+tm(x,\x)}$ here: If $u\in L^{2,k+tm(x,\x)}$, then
\begin{align}
u\in& L^{2,k-t}, \text{microlocally near} \{|x|,|\x|>R, |\x|\sim |x|^{\a}, x\cdot \x\sim |x||\x|\}\label{inc}\\
u\in& L^{2,k+t}, \text{microlocally near} \{|x|,|\x|>R, |\x|\sim |x|^{\a}, x\cdot \x\sim -|x||\x|\}\label{out}
\end{align}
for large $R>0$.
The following theorem is an analog of \cite[Theorem 1.4]{FS}.

\begin{thm}\label{thm1}
\noindent\begin{itemize}
\item[$(i)$] Let $t\ne 0$ and $z\in \mathbb{C}$. We define
\begin{align*}
D_{tm}=\{u\in L^{2,\frac{\a-1}{2}+tm(x,\x)}  \mid (P-z)u\in L^{2,\frac{1-\a}{2}+tm(x,\x)} \}.
\end{align*}
Then 
\begin{align}\label{Pmap}
P-z: D_{tm}\to L^{2,\frac{1-\a}{2}+tm(x,\x)}
\end{align}
is a Fredholm operator and coincides with the closure of $(P-z)$ with domain $\mathcal{S}(\re^n)$ with respect to its graph norm. 

\item[$(ii)$] There exists a discrete subset $T_{\a,t}\subset \mathbb{C}$ such that $(\ref{Pmap})$ is invertible for $\mathbb{C}\setminus T_{\a,t}$. 
\end{itemize}
\end{thm}
\begin{rem}
By the standard radial point estimates and the propagation of singularities, it follows that $T_{\a,t}=T_{\a,\sgn t}$ is independent of $|t|$ and $T_{\a,t}\subset \mathbb{C}_{-\sgn t}=\{-(\sgn t)\Im z\geq 0\}$. Moreover, this theorem is true for $0<\a\leq 1$ if we replace $z\in \mathbb{C}$ above by $z\in \mathbb{C}_{\sgn t}$ (though $D_{tm}$ depends on $z$). We leave their proofs to future work.
\end{rem}

This theorem also gives the bijectivity of $P-z$ in the usual weighted $L^2$-spaces: Suppose  $z\in \mathbb{C}\setminus T_{\a,t}$.  For any $f\in L^{(1-\a)/2+\e}$ with $\e>0$, there exists  a unique solution $u\in L^{2,(\a-1)/2-\e}$ to the equation
\begin{align*}
\begin{cases}
(P-z)u=f,\,\, \text{in the distributional sense},\\
u\,\, \text{is outgoing if the signature is $+$ and incoming if the signature is $-$},
\end{cases}
\end{align*}
where "$u$ is outgoing" says that $(\ref{out})$ holds with $k=(\a-1)/2$ and $t=\e$ and "$u$ is incoming" says that $(\ref{inc})$ holds with $k=(\a-1)/2$ and $t=-\e$.

Moreover, we construct non-trivial $L^2$ solutions to $Pu=zu$.

\begin{thm}\label{thm2}
Let $\a>1$ and $0<|t|<1/2$. For $z\in \mathbb{C}\setminus T_{\a,t}$, there exists $u\in L^2\setminus \{0\}$ such that $Pu=zu$.
\end{thm}
\begin{rem}
As is proved in Proposition \ref{many}, it follows that there are many eigenfucntions associated with $z\in \mathbb{C}\setminus T_{\a,t}$.
\end{rem}

From Theorem \ref{thm2} and the standard criterion for essential self-adjointness \cite[Corollary after Theorem VIII.3]{RS}, we conclude that $P$ is not essentially self-adjoint if $\a>1$.

\begin{cor}
Suppose $\a>1$. Then $P=P_{\a}$ is not essentially self-adjoint on $C_c^{\infty}(\re^n)$ and $\mathcal{S}(\re^n)$.
\end{cor}

The repulsive Schr\"odinger operator $P=P_{\a}$ for large $\a$ is expected to have the same structure to the Laplace operator on a bounded open set in $\re^n$. For a bounded open set $\O$, it easily follows that the inclusion $H^2_0(\O)\hookrightarrow L^2(\O)$ is compact. Here we note that $H^2_0(\O)$ is the minimum domain of $-\Delta|_{C_c^{\infty}(\O)}$. For the repulsive Schr\"odinger operator, we prove the similar result.

\begin{thm}\label{thm3}
Define a Banach space
\begin{align*}
D^{\a}_{\mathrm{min}}=\{u\in L^2(\re^n) \mid Pu\in L^2(\re^n),\, \exists u_n\in C_c^{\infty}(\re^n)\, \, u_k\to u, Pu_k\to Pu\, \text{in}\, L^2(\re^n) \}
\end{align*}
with its graph norm. Then the inclusion $D^{\a}_{\mathrm{min}}\hookrightarrow L^2$ is compact.
\end{thm}
\begin{rem}
$D^{\a}_{\mathrm{min}}$ coincides with the minimal domain of $P|_{C_c^{\infty}(\re^n)}$, that is the domain of the closure of $P|_{C_c^{\infty}(\re^n)}$.
\end{rem}

\begin{cor}\label{disc1}
Let $n=1$ and $P_U$ be a self-adjoint extension of $P$. Then there exists $\{\l_k\}_{k=1}^{\infty}\subset \re$ such that $\s(P_U)=\s_d(P_U)=\{\l_k\}_{k=1}^{\infty}$ and $|\l_k|\to \infty$ as $k\to \infty$, where $\s(P_U)$ is the spectrum of $P_U$ and $\s_d(P_U)$ is the discrete spectrum of $P_U$.
\end{cor}

\begin{rem}
For a relatively bounded open interval $I\subset \re$, it is proved that each self-adjoint extension of $-\Delta|_{C_c^{\infty}(I)}$ has a discrete spectrum by mimicking the proof of Corollary \ref{disc1}. However, in the case of $n\geq 2$, the situation is dramatically different.
In fact, we consider the Klein Laplacian ($-\Delta$ with domain $\{u\in L^2(\O)\mid \Delta u=0\}+H^2_0(\O)$) for the bounded domain with smooth boundary $\pa \O$. The Klein Laplacian has a nonempty essential spectrum for $n\geq 2$. In fact, we note that any $L^2$ harmonic functions on $\O$ lies in the domain of the Klein Laplacian. Since restrictions of harmonic functions on $\re^n$ to $\O$ are $L^2$ harmonic functions on $\O$ and since the dimension of the set of all harmonic functions for $n\geq 2$ is infinite, we conclude that $0$ is the eigenvalue with infinite multiplicity. In this way, it follows that the essential spectrum is not empty. 
\begin{rem}
 As an analogy to $-\Delta$ on $\O$, we naturally propose the following problems:
\begin{itemize}
\item Does there exist a distinguish self-adjoint extension of $P$ (such as the Friedrichs extension of $-\Delta|_{C_c^{\infty}(\O)}$ in the case of $-\Delta$ on $\O$)?

\item How is the structure of the self-adjoint extension of $P$? (More concretely, does there exist a self-adjoint extension of $P$ which has a discrete spectrum?)
\end{itemize}

\end{rem}

\end{rem}

We fix some notations. $\mathcal{S}(\re^n)$ denotes the set of all rapidly decreasing functions on $\re^n$ and $\mathcal{S}'(\re^n)$ denotes the set of all tempered distributions on $\re^n$. We use the weighted Sobolev space: $L^{2,l}=\jap{x}^{-l}L^2(\re^n)$, $H^k=\jap{D}^{-k}L^2(\re^n)$ and $H^{k,l}=\jap{x}^{-l}\jap{D}^{-k}L^2(\re^n)$ for $k,l\in \re$. For Banach spaces $X,Y$, $B(X,Y)$ denotes the set of all linear bounded operators form $X$ to $Y$. For a Banach space $X$, we denote the norm of $X$ by $\|\cdot\|_{X}$. If $X$ is a Hilbert space, we write the inner metric of $X$ by $(\cdot, \cdot)_{X}$, where $(\cdot, \cdot)_{X}$ is linear with respect to the right variable. We also denote $\|\cdot\|_{L^2}=\|\cdot\|_{L^2(\re^n)}$ and $(\cdot,\cdot)_{L^2}=(\cdot, \cdot)_{L^2(\re^n)}$. We denote the distribution pairing by $<\cdot, \cdot>$. For $I\subset \re$, we denote $I_{\pm}=\{z\in \mathbb{C}\mid \Re z\in I, \pm \Im z\geq 0 \}$. We denote $\jap{x}=(1+|x|^2)^{1/2}$ for $x\in \re^n$. Set
\begin{align*}
\mathbb{C}_{\pm}=\{z\in \mathbb{C}\mid \pm\Im z\geq 0\}.
\end{align*}

\noindent
\textbf{Acknowledgment.}  
This work was supported by JSPS Research Fellowship for Young Scientists, KAKENHI Grant Number 17J04478 and the program FMSP at the Graduate School of Mathematics Sciences, the University of Tokyo. The author would like to thank Shu Nakamura for pointing out mistakes of the first draft and for useful comments. The author also would like to thank Kyohei Itakura, Kenichi Ito and Kentaro Kameoka for helpful discussions.

\section{Preliminary}

\subsection{Notations and cut-off functions}\label{subcut}
In this subsection, we fix some notations and define cut-off functions which are used in this paper many times.

Let $\chi\in C_c^{\infty}(\re, [0,1])$ such that
\begin{align*}
\chi(t)= \begin{cases}
1,& |t|\leq 1, \\
0,& |t|\geq 2.
\end{cases}
\end{align*}
For $R, L\geq 1$ and $0<r\leq  1$, set $\bar{\chi}=1-\chi$ and 
\begin{align}
a_{r,R}(x,\x)=&\bar{\chi}(|x|/R)\bar{\chi}(|\x|/R)\chi(|\x|^2-|x|^{2\a})/r(|\x|^2+|x|^{2\a})),\label{cut1}\\
a_R(x,\x)=&a_{R^{-1}, R}(x,\x),\,\,b_L(x,\x)=\chi(|x|/L)\chi(|\x|/L)\label{cut2}.
\end{align}

We often use the symbol
\begin{align*}
\y(x,\x)=\frac{x\cdot \x}{|x||\x|}.
\end{align*}

\subsection{Pseudodifferential operators}
Set 
\begin{align*}
S^{k,l}=\{a\in C^{\infty}(\re^n)\mid |\pa_{x}^{\c_1}\pa_{\x}^{\c_2}a(x,\x)|\leq C_{\c_1,\c_2}\jap{x}^{l-|\c_1|}\jap{\x}^{k-|\c_2|} \}.
\end{align*}
We denote the Weyl quantization of $a\in C^{\infty}(\re^{2n})\cap \mathcal{S}'(\re^n)$ by $\Op(a)$:
\begin{align*}
\Op(a)u(x)=\frac{1}{(2\pi)^n}\int_{\re^n}\int_{\re^n}e^{i(x-y)\cdot \x}a(\frac{x+y}{2},\x)u(y)dyd\x.
\end{align*}
We denote $\Op S=\Op(S)$ for $S\subset \mathcal{S}'(\re^{2n})$.
Recall that if $a$ is real-valued, then $\Op(a)$ is formally self-adjoint with respect to the metric on $L^2(\re^n)$. We denote the composition of the Weyl calculus by $\#$:
\begin{align*}
\Op(a\# b)=\Op(a)\Op(b).
\end{align*}

The following lemmas is easily proved by a standard $\e/3$-argument.

\begin{lem}\label{b_L}
Let $b_L$ as in subsection \ref{subcut} and $Q\in S^{k,l}$ for some $k,l\in \re$. Then the symbol of $[Q, \Op(b_L)]$ is uniformly bounded in $S^{k-1,l-1}$ with respect to $L\geq 1$ and converges to $0$ in $S^{k-1+\e, l-1+\e}$ as $L\to \infty$ for any $\e>0$.
\end{lem}

The following proposition is proved by the standard parametrix construction and Assumption \ref{assb}. We omit its proof.

\begin{prop}[Elliptic estimate]\label{elliptic}
Let $z\in \mathbb{C}$, $k,l\in \re$, $N>0$ and $k_1,l_1\geq 0$ with $k_1+l_1\leq 2$. For $R, M\geq 1$ and $\c>1$, set
\begin{align*}
\O_{loc}=&\{(x,\x)\in \re^{2n} \mid |x|<M\}\cup \{(x,\x)\in \re^{2n} \mid |\x|<M\},\\
\O_{R,\c,1}=&\{(x,\x)\in \re^{2n} \mid |x|>R, |\x|>R, |\x|>\c |x|^{\a}\},\\
\O_{R,\c,2}=&\{(x,\x)\in \re^{2n} \mid |x|>R, |\x|>R, |x|^{\a}>\c |\x|\}.
\end{align*}
Let $\c>1$. There exists $R_1>0$ such that if $R\geq R_1$ and $a, a_1\in S^{0,0}$ are supported in $\O_{loc}\cup \O_{R,\c,1}\cup \O_{R,\c,2}$ and $\inf_{\supp a}|a_1|>0$, then there exists $C>0$ such that for $u\in H^{-N,-N}$ with $\Op(a_1) Pu\in H^{k,l}$, we have $\Op(a)u\in H^{k+k_1, l+\a l_1}$ and 
\begin{align*}
\|\Op(a)u\|_{H^{k+k_1,l+\a l_1}}\leq C\|\Op(a_1)(P-z)u\|_{H^{k,l}}+C\|u\|_{H^{-N,-N}}.
\end{align*}
Here the constant $C>0$ is locally uniformly in $\Re z\in \re$.
\end{prop}

The next lemma follows from a simple observation; $|\x|\sim |x|^{\a}$ on $\supp a_R$. 
\begin{lem}\label{suppell}
Let $k,l\in \re$. If $u\in H^{k,l}$, then $\Op(a_R)u\in H^{k+M, l-\a M}$ for $M\in \re$, where $a_R$ is as in subsection \ref{subcut}.
\end{lem}
\begin{proof}
First, suppose $|M|\leq 1$. 
By a support property of $a_R$, we have 
\begin{align*}
\jap{\x}^{M}\#\jap{x}^{-\a M}\# a_R\in S^{0,0}
\end{align*}
Then
\begin{align*}
\|\Op(a_R)u\|_{H^{k+M, l-\a M}}\leq& C\| \Op(\jap{\x}^{M}\# \jap{x}^{-\a M}\#(a_R))u\|_{H^{k,l}}\\
\leq&C'\|u\|_{H^{k,l}}
\end{align*}
with some $C>0$.
\end{proof}

\section{Proof of Theorem \ref{thm1} $(i)$}

For $k\in \re$, we set
\begin{align*}
S_{\a}^k=\bigcup_{l\in \re}S^{l,k-\a l}.
\end{align*}

\subsection{Construction of an escape function}\label{subsecesc}

Take $\rho\in C^{\infty}(\re,[0,1])$ such that 
\begin{align}
&\rho(t)=1,\,\, \text{if}\,\, |t|\geq 1/2,\,\, \rho'(t)\geq 0,\,\, \rho'(t)\geq C_1\geq  C_2|\rho(t)|,\,\, \text{if}\,\, |t|\leq 1/4,\label{ro21}\\
&\rho'(t)\leq C_3\leq  C_4|\rho(t)|,\,\, \text{if}\,\, |t|\geq 1/4,\,\, t\rho(t)\geq 0.\label{ro22}
\end{align}
We define
\begin{align*}
m(x,\x)=m_{R}(x,\x)=-\rho(\y(x,\x))a_R(x,\x)^2,
\end{align*}
where $\y(x,\x)=x\cdot \x/|x||\x|$ and $a_R$ is as in $(\ref{cut2})$. Moreover, we set
\begin{align*}
\O_R=\{(x,\x)\in \re^{2n}\mid |x|>R,\, |\x|>R,\, -\frac{1}{2R}<\frac{|\x|^2-|x|^{2\a}}{|\x|^2+|x|^{2\a}}<\frac{1}{2R} \}.
\end{align*}

\begin{lem}\label{escape}
There exists $R_0\geq 1$ such that if $R\geq R_0$, then
\begin{align*}
H_p(m\log\jap{x})(x,\x)\leq -C\jap{x}^{\a-1}a_R(x,\x)^2-e(x,\x),
\end{align*}
where $e(x,\x)=\rho(\y(x,\x))(H_pa_R)(x,\x)\log\jap{x}\in S_{\a}^{\a-1+0}$.
\end{lem}

\begin{proof}
We learn
\begin{align*}
H_p(\rho(\y)\log\jap{x})\geq& 2(\y\rho(\y))|x||\x|\jap{x}^{-2}+(H_{p_0}\y)\rho'(\y)\log\jap{x}\\
&-C|\rho(\y)|\jap{x}^{\a-1-\m}-C|\rho'(\y)|\jap{x}^{\a-1-\m}\log\jap{x}.
\end{align*}
Note that the first line of the right hand side is positive for $(x,\x)\in \O_R$. Moreover, we observe that $|\x|\sim |x|^{\a}$ on $\O_R$ if $R$ is large enough. For $|\y(x,\x)|\geq 1/4$, it follows 
\begin{align*}
H_p(\rho(\y)\log\jap{x})\geq& 2(\y\rho(\y))|x||\x|\jap{x}^{-2}-C|\rho(\y)|\jap{x}^{\a-1-\m}\\
&-C|\rho'(\y)|\jap{x}^{\a-1-\m}\log\jap{x}\\
\geq &C\jap{x}^{\a-1}|\rho(\y)|-C\jap{x}^{\a-1-\m}\log\jap{x}|\rho(\y)|\geq C\jap{x}^{\a-1}.
\end{align*}
by $(\ref{ro21})$ and $(\ref{ro22})$. For $|\y(x,\x)|\leq 1/4$, we have
\begin{align*}
H_p(\rho(\y)\log\jap{x})\geq&(H_{p_0}\y)\rho'(\y)\log\jap{x}-C|\rho(\y)|\jap{x}^{\a-1-\m}\\
&-C|\rho'(\y)|\jap{x}^{\a-1-\m}\log\jap{x}\\
\geq& C\rho'(\y)\jap{x}^{\a-1}\log\jap{x} -C|\rho'(\y)|\jap{x}^{\a-1-\m}\log\jap{x}\geq C\jap{x}^{\a-1}.
\end{align*}
Thus we complete the proof.
\end{proof}

\subsection{Fredholm properties}

Let $m=m_{R_0}$ be as in subsection \ref{subsecesc}, where $R_0$ is as in Lemma \ref{escape}. Moreover, we set $k_{\a}=(\a-1)/2$. Let $S^{k,tm(x,\x)+l}$ be as in Definition \ref{def1app} and let $\tilde{G}_{k_{\a},tm}(x,\x)=\jap{x}^{k_{\a}+tm(x,\x)}+ S^{-\infty, -\infty}$ such that $\Op(\tilde{G}_{k_{\a},tm}):\mathcal{S}(\re^n)\to \mathcal{S}(\re^n)$ is invertible. Existence of such $\tilde{G}_{k_{\a},tm}$ is proved in Lemma \ref{varinv} (see also (\ref{Ginvdef})). Moreover, the variable order weighted $L^2$-space $L^{2, k_{\a}+tm(x,\x)}$ is defined by
\begin{align*}
L^{2, k_{\a}+tm(x,\x)}=\Op(\tilde{G}_{k_{\a},tm})^{-1}L^2.
\end{align*}
By Lemma \ref{varequi} $(ii)$, we have
\begin{align*}
L^{2, k_{\a}+tm(x,\x)}=\Op(\tilde{G}_{0,tm})^{-1}L^{2,k_{\a}}.
\end{align*}

For $t\neq 0$ and $z\in \mathbb{C}_{\pm}$, we set 
\begin{align}\label{unitary}
P_{tm}(z)=\Op(\tilde{G}_{0, tm})(P-z)\Op(\tilde{G}_{0,tm})^{-1}.
\end{align}
We note that the operator $P$ on $L^{2, k_{\a}+tm(x,\x)}$ is unitary equivalent to $P_{tm}$ on $L^{2,k_{\a}}$. This is why we study the Fredholm property of $P_{tm}(z)$ instead of $P$ in order to prove Theorem \ref{thm1}.
By the asymptotic expansion, we have
\begin{align*}
P_{tm}(z)=P-z+it\Op(H_p(m\log\jap{x}))+\Op S^{0,-2+0}
\end{align*}
since $|\x|\sim |x|^{\a}$ on $\supp m$ and $\tilde{G}_{0, tm}=\jap{x}^{tm(x,\x)}+S^{-\infty, -\infty}$. 

\begin{lem}\label{mGar}
We have
\begin{align*}
-(u, \Op(H_p(m\log\jap{x}))u)_{L^2}\geq \|\Op(a_R)u\|_{L^{2,\frac{\a-1}{2}}}^2 -C\|u\|_{L^{2,-1+0}}^2+(u,\Op(e)u)_{L^2}
\end{align*}
for $u\in \mathcal{S}(\re^n)$.
\end{lem}

\begin{proof}
By the construction, $m$ is supported in $\supp a_R$. Hence we have 
\begin{align*}
H_p(m\log\jap{x})\in S^{1,-1+0}.
\end{align*}
By Lemma \ref{escape} and the sharp  G$\rm{\mathring{a}}$rding inequality, we obtain the above inequality.
\end{proof}

\begin{lem}\label{Ddense}
Set $\tilde{D}_{tm}(z)=\{u\in L^{2,(\a-1)/2}\mid P_{tm}(z)u\in L^{2,(1-\a)/2}\}$. We consider $\tilde{D}_{tm}(z)$ as a Banach space with its graph norm. Then $\mathcal{S}(\re^n)$ is dense in $\tilde{D}_{tm}(z)$.
\end{lem}
\begin{proof}
Let $u\in\tilde{D}_{tm}(z)$. We recall that $b_{L}(x,\x)=\chi(|x|/L)\chi(|\x|/L)$ is as in $(\ref{cut2})$. Since $\Op(b_L)u\to u$ in $L^{2,(\a-1)/2}$ and $\Op(b_L)P_{tm}(z)u\to P_{tm}(z)u$ in $L^{2,(1-\a)/2}$, it suffices to prove that $[P_{tm},\Op(b_L)]u\to 0$ in $L^{2,(1-\a)/2}$.
We learn
\begin{align*}
\|[P_{tm}(z), \Op(b_L)]u\|_{L^{2,(1-\a)/2}}\leq& \|[P_{tm}(z), \Op(b_L)]\Op(a_R)u\|_{L^{2,(1-\a)/2}}\\
&+\|[P_{tm}(z), \Op(b_L)](1-\Op(a_R))u\|_{L^{2,(1-\a)/2}}.
\end{align*}
Since $|\x|\sim |x|^{\a}$ on $a_R$, it follows that $[P_{tm}(z),\Op(b_L)]\Op(a_R)$ is uniformly bounded in $S^{0,\a-1}$ and converges to $0$ in $S^{0,(\a-1)/2+0}$. Lemma \ref{b_L} and $u\in L^{2,(\a-1)/2}$ imply
\begin{align*}
\limsup_{L\to \infty}\|[P_{tm}(z), \Op(b_L)]\Op(a_R)u\|_{L^{2,(1-\a)/2}}=0.
\end{align*}
Moreover, since  $u\in L^{2,(\a-1)/2}$ with $P_{tm}(z)u\in L^{2, (1-\a)/2}$, then the elliptic estimates (Proposition \ref{elliptic}) implies $(1-\Op(a_R))u\in H^{k_1, (1-\a)/2+\a l_1}$ for $k_1,l_1\geq 0$ with $k_1+l_1\leq 2$. In particular,
\begin{align*}
(1-\Op(a_R)) u\in \bigcap_{j=1}^2H^{j,\frac{\a+1}{2}+(j-1)\a}.
\end{align*}
Since $[P_{tm}(z), \Op(b_L)]$ is uniformly bounded in $\sum_{j=0}^2 S^{1-j, j\a-1}$ and converges to $0$ in $\sum_{j=0}^2 S^{1-j+\e, j\a-1+\e}$ for any $\e>0$, then Lemma \ref{b_L} gives
\begin{align*}
\limsup_{L\to \infty}\|[P_{tm}(z), \Op(b_L)](1-\Op(a_R))u\|_{L^{2,(1-\a)/2}}=0.
\end{align*}
This completes the proof.
\end{proof}

\begin{prop}
Let $I\subset \re$ be a relativity compact interval. Then there exists $C>0$ such that for $z\in I_{\sgn t}$ we have
\begin{align}
\|u\|_{L^{2,(\a-1)/2}}\leq C\|P_{tm}(z)u\|_{L^{2, (1-\a)/2}}+C\|u\|_{H^{-N,-N}},\,\, u\in \tilde{D}_{tm}(z),\label{13}\\
\|u\|_{L^{2,(\a-1)/2}}\leq C\|P_{tm}(z)^*u\|_{L^{2, (1-\a)/2}}+C\|u\|_{H^{-N,-N}},\,\, u\in\tilde{D}_{tm}(\bar{z}).  \label{14}
\end{align}
Moreover, $(\ref{13})$ and $(\ref{14})$ hold for $z\in I_{-\sgn t}$ though the constant $C>0$ depends on $\Im z$.
\end{prop}

\begin{proof}
First, we assume $z\in I_{\sgn t}$.
We  prove $(\ref{13})$ only. Since $P_{tm}(z)^*=(P-z)^*-it\Op(H_p(m\log\jap{x}))+\Op S^{0,-2+0}$ holds, $(\ref{14})$ is similarly proved. By Lemma \ref{Ddense}, we may assume $u\in \mathcal{S}(\re^n)$.
By Lemma \ref{mGar} and $t\Im z\geq 0$, then
\begin{align*}
-(\sgn t)\Im(u,P_{tm}(z)u)_{L^2}\geq& \|\Op(a_R)u\|_{L^{2,(\a-1)/2}}^2-C\|u\|_{L^{2,-1+0}}^2+(u, \Op(e)u)_{L^2}
\end{align*}
for $u\in \mathcal{S}(\re^n)$.
Since $t\Im z\geq 0$, then we have
\begin{align}
\|\Op(a_R)u\|_{L^{2,(\a-1)/2}}^2\leq& C\|P_{tm}(z)u\|_{L^{2,(1-\a)/2}}\|u\|_{L^{2,(\a-1)/2}}+C\|u\|_{L^{2,-1+0}}^2\label{11}\\
&+|(u, \Op(e)u)_{L^2}|.\nonumber
\end{align}
By the elliptic estimate (Proposition \ref{elliptic}) and the interpolation estimate, we have
\begin{align}\label{12}
\|&(1-\Op(a_R))u\|_{L^{2,(\a-1)/2}}^2+\|u\|_{L^{2,-1+0}}^2+|(u, \Op(e)u)_{L^2}|\\
&\leq C\|\Op(a_R)u\|_{L^{2,-1+0}}^2+C\|P_{tm}(z)u\|_{L^{2,(1-\a)/2}}^2+C\|u\|_{H^{-N,-N}}^2\nonumber\\
&\leq \frac{1}{2}\|\Op(a_R)u\|_{L^{2,(\a-1)/2}}^2+C\|P_{tm}(z)u\|_{L^{2,(1-\a)/2}}^2+C\|u\|_{H^{-N,-N}}^2.\nonumber
\end{align}
By using $(\ref{11})$, $(\ref{12})$ and the Cauchy-Schwarz inequality, we obtain $(\ref{13})$ for $u\in \mathcal{S}(\re^n)$.

Next, we prove that $(\ref{13})$ and $(\ref{14})$ hold for $z\in I_{-\sgn t}$ though the constant $C>0$ depends on $\Im z$. In fact, since $(\a-1)/2>(1-\a)/2$, then the elliptic estimate and the interpolation inequality implies that for any $\e_1>0$,
\begin{align*}
|\Im z|\|u\|_{L^{2, (1-\a)/2}}^2\leq& \e_1\|u\|_{L^{2, (\a-1)/2}}^2+C\|u\|_{L^{2,-N-N\a}}^2\\
\leq & \e_1\|u\|_{L^{2, (\a-1)/2}}^2+C\|P_{tm}(z)\|_{L^{2,(1-\a)/2}}^2+C\|u\|_{H^{-N,-N}}^2.
\end{align*}
Taking $\e_1>0$ small enough and use $(\ref{13})$ and $(\ref{14})$ for $\bar{z}$, we obtain $(\ref{13})$ for $z\in I_{-\sgn t}$. 
\end{proof}

\begin{rem}\label{Fredinv}
Suppose $t\geq 0$. If $\Im z$ is large enough, then
\begin{align}\label{15}
\|u\|_{L^{2,(\a-1)/2}}\leq C\|P_{tm}(z)u\|_{L^{2,(\a-1)/2}}.
\end{align}
In fact, in $(\ref{11})$, we have a stronger bound:
\begin{align*}
\|\Op(a_R)u\|_{L^{2,(\a-1)/2}}^2+\Im z\|u\|_{L^2}^2\leq (\text{RHS of $(\ref{11})$}).
\end{align*}
Hence the argument after $(\ref{11})$ implies
\begin{align*}
(1+\e)\|u\|_{L^{2,(\a-1)/2}}^2+\Im z\|u\|_{L^2}^2\leq C\|P_{tm}(z)u\|_{L^{2,(\a-1)/2}}^2+C\|u\|_{H^{-N,-N}}^2.
\end{align*}
We use the trivial bounds $\|u\|_{H^{-N,-N}}\leq \|u\|_{L^2}$ $\|u\|_{H^{-N,-N}}\leq \|u\|_{L^{2,(\a-1)/2}}$ and we obtain $(\ref{15})$. Similarly, for $t\leq 0$, $(\ref{15})$ holds if $-\Im z$ is large enough.
\end{rem}

\begin{cor}\label{frecor}
A map
\begin{align}\label{P_tmmap}
P_{tm}(z):\tilde{D}_{tm}(z)\to L^{2,(1-\a)/2}
\end{align}
is a Fredholm operator. Moreover, if $t\Im z\geq 0$ holds and $|\Im z|$ is large enough, then $P-z$ is invertible. 
Furthermore, $(\ref{P_tmmap})$ is an analytic family of Fredholm operators with index zero.  Moreover, there exists a discrete set $T_{\a,t}\subset \mathbb{C}$ such that $(\ref{P_tmmap})$ is invertible for $z\in \mathbb{C}\setminus T_{\a,t}$.
\end{cor}
\begin{rem}\label{remFre2}
Remark \ref{Fredinv} implies that $P_{tm}(z)$ is invertible for $t\geq0$ and for large $\Im z >0$. In fact, the injectivity of $P_{tm}(z)$ follows from $(\ref{15})$ and the surjectivity follows from the injectivity of $P_{tm}(z)^*$.
\end{rem}

\begin{proof}
First, we prove that $\Ker P_{tm}(z)<\infty$ is of finite dimension and $\Ran P_{tm}(z)$ is closed. Let a bounded sequence $u_k \in\tilde{D}_{tm}(z)$ such that $P_{tm}(z)u_k$ is convergent in $L^{2,(1-\a)/2}$. Due to \cite[Proposition 19.1.3]{Ho}, it suffices to prove that $u_k$ has a convergent subsequence in $\tilde{D}_{tm}(z)$. It easily follows from $(\ref{13})$ and the compactness of the inclusion $L^{2,(\a-1)/2}\subset H^{-N,-N}$. 

Next, we prove that the cokernel of $P_{tm}(z)$ is of finite dimension. To do this, it suffices to prove that the kernel of $P_{tm}(z)^*: L^{2,(\a-1)/2}\to \tilde{D}_{tm}(z)^*$ is of finite dimension.  
By definition, we have
\begin{align*}
\Ker P_{tm}(z)^*=&\{u\in L^{2,(\a-1)/2}\mid  (u,P_{tm}(z)w)_{L^2}=0,\, \forall w\in \tilde{D}_{tm}(z) \}\\
=&\{u\in L^{2,(\a-1)/2}\mid  (u,P_{tm}(z)w)_{L^2}=0,\, \forall w\in \mathcal{S}(\re^n) \},
\end{align*}
where we use Lemma \ref{Ddense} in the second line. If $u\in L^{2,(\a-1)/2}$ satisfies $P_{tm}(z)^*u=0$, then this equality holds in the distributional sense. The claim follows same as in the first half part of this proof.

The invertibility of $(\ref{P_tmmap})$ when $t\Im z\geq 0$ and when $|\Im z|$ is large follows from Remark \ref{Fredinv} and its dual statement. The analytic Fredholm theorem  \cite[Theorem D.4]{Z} imply existence of $T_{\a,t}$ as above.

\end{proof}

\begin{proof}[Proof of Theorem \ref{thm1}] Theorem \ref{thm1} follows from $(\ref{unitary})$ and Corollary \ref{frecor}.

\end{proof}

\section{Proof of Theorem \ref{thm2}}

\subsection{Outgoing/incoming parametrices}

In this subsection, we construct outgoing/incoming parametrices of a solution to $Pu=zu$. 
Set
\begin{align*}
S^{k}(\re^n)=\{a\in C^{\infty}(\re^n\setminus\{0\})\mid |\pa_{x}^{\b}a(x)|\leq C_{\b}\jap{x}^{k-|\b|},\,\, \text{for}\,\, |x|>1\}.
\end{align*}
Moreover, we frequently use the following notation:
\begin{align*}
\hat{x}=x/|x|.
\end{align*}

The main result of this subsection is the following theorem. 

\begin{thm}\label{parathm}
Fix a signature $\pm$ and $a\in C^{\infty}(\mathbb{S}^{n-1})$. Then there exists $\f_{\pm}\in S^{1+\a}(\re^n)$ such that
\begin{align*}
&\f_{\pm}\mp\frac{|x|^{1+\a}}{1+\a} \mp z\frac{|x|^{1-\a}}{2(1-\a)}\in S^{1+\a-\m}(\re^n),\,\,\Im(\f_{\pm}\mp z\frac{|x|^{1-\a}}{2(1-\a)})\in S^0(\re^n),\\
&e^{-i\f}(-\Delta-|x|^{2\a}+\Op(V)-z)(e^{i\f}b)\in S^{-\frac{n+1-\a}{2}-\m}(\re^n),
\end{align*}
where $b(x)=|x|^{-\frac{n-1+\a}{2}}\bar{\chi}(|x|/R)a(\hat{x})\in S^{-\frac{n-1+\a}{2}}(\re^n)$ and $\hat{x}=x/|x|$.
\end{thm}

Theorem \ref{parathm} is proved by Propositions \ref{paraprop} and \ref{eikonal} below.

\begin{prop}\label{paraprop}
Fix a signature $\pm$, $z\in \mathbb{C}$ and $a\in C^{\infty}(\mathbb{S}^{n-1})$. Set $b(x)=|x|^{-\frac{n-1+\a}{2}}\bar{\chi}(|x|/R)a(\hat{x})\in S^{-\frac{n-1+\a}{2}}(\re^n)$. Let $\f_{\pm,z}\in S^{1+\a}(\re^n)$ be satisfying
\begin{align*}
\f_{\pm,z}\mp\frac{|x|^{1+\a}}{1+\a}\mp z\frac{|x|^{1-\a}}{2(1-\a)} \in S^{1+\a-\m}(\re^n),\,\, \Im(\f_{\pm,z}\mp z\frac{|x|^{1-\a}}{2(1-\a)})\in S^0(\re^n).
\end{align*}
Then we have
\begin{align*}
&e^{-i\f_{\pm,z}}(-\Delta-|x|^{2\a}+\Op(V)-z)(e^{i\f_{\pm,z}}b)\\
&=((\nabla\f)^2-|x|^{2\a}+V(x,\nabla\f_{\pm.z}(x))-z)b(x)+ S^{-\frac{n+1-\a}{2}-\m}(\re^n).
\end{align*}
\end{prop}

Proposition \ref{paraprop} directly follows from Lemmas \ref{paramain} and \ref{paraper} below.
\begin{lem}\label{paramain}
Fix a signature $\pm$ and $z\in \mathbb{C}$. Let $\f_{\pm,z}$ and $b$ be as in the above proposition. 
Then
\begin{align*}
e^{-i\f_{\pm,z}}(-\Delta-|x|^{2\a}-z)(e^{i\f_{\pm,z}}b)=((\nabla\f_{\pm,z})^2-|x|^{2\a}-z)b+S^{-\frac{n+1-\a}{2}-\m}(\re^n).
\end{align*}
\end{lem}

\begin{proof}
Set $k=-\frac{n-1+\a}{2}$, then we note $k+\a-\m-1=-\frac{n+1-\a}{2}-\m$. We write $\f=\f_{\pm}$ and $\f_0=\f_{0,\pm}=\pm|x|^{1+\a}/(1+\a)$. By a simple calculation, we have
\begin{align*}
e^{-i\f}(-\Delta-|x|^{2\a}-z)(e^{i\f}b)=&((\nabla\f)^2-|x|^{2\a}-z)b-i(2\nabla\f\cdot\nabla b+(\Delta\f)b)-(\Delta b).
\end{align*}
Due to $b\in S^k$ and $\f-\f_0\in S^{1+\a-\m}$, we observe
\begin{align*}
\Delta b,\,2\nabla(\f-\f_0)\cdot \nabla b+\Delta(\f-\f_0)b\in S^{k+\a-\m-1}(\re^n).
\end{align*}
Thus, it suffices to prove 
\begin{align*}
2\nabla \f_0\cdot \nabla b+(\Delta\f_0)b\in S^{k+\a-\m-1}(\re^n).
\end{align*}
Since $\nabla\f_0=\pm |x|^{\a-1}x,\,\,\Delta\f_0=\pm (n-1+\a)|x|^{\a-1}$, we obtain
\begin{align*}
2\nabla \f_0\cdot \nabla b+(\Delta\f_0) b =&\pm(2|x|^{\a}\pa_rb(x)+(n-1+\a)|x|^{n-1+\a} b)\\
=&\pm \frac{2}{R}|x|^{k+\a}a(\hat{x})(\bar{\chi})'(|x|/R) \in C_c^{\infty}(\re^n) \subset S^{k+\a-\m-1}(\re^n).
\end{align*}
\end{proof}

\begin{lem}\label{paraper}
Let $k\in \re$, $\f\in S^{1+\a}(\re^n)$ and $b\in S^k(\re^n)$. Set $\g(x,y)=\int_0^1\nabla\f(tx+(1-t)y)dt$. Then
\begin{align*}
e^{-i\f(x)}\Op(V)e^{i\f}b(x)=V(x,\nabla\f(x))b(x)+L(x),
\end{align*}
where $L\in S^{k+\a-\m-1}(\re^n)$ is defined by
\begin{align*}
L(x)=D_y(\pa_{\x}V(x,\g(x,y)) b(x))|_{x=y}+(D_y^2(\pa_{\x}^2V(\frac{x+y}{2},0) b(y))|_{x=y}.
\end{align*}
\end{lem}

\begin{proof}
By a simple calculation, we have
\begin{align*}
e^{-i\f(x)}\Op(V)(e^{i\f}b)(x)=&\frac{1}{(2\pi)^n}\int_{\re^{2n}}e^{i(x-y)\cdot\x-i(\f(x)-\f(y))}V(\frac{x+y}{2}, \x)b(y)dyd\x\\
=&\frac{1}{(2\pi)^n}\int_{\re^{2n}}e^{i(x-y)\cdot\x}V(\frac{x+y}{2}, \x+\g(x,y))b(y)dyd\x\\
=&V(x,\nabla\f(x))b(x)+L(x),
\end{align*}
where
\begin{align*}
L(x)=\frac{1}{(2\pi)^n}\int_{\re^{2n}}e^{i(x-y)\cdot\x}(V(\frac{x+y}{2}, \x+\g(x,y))-V(\frac{x+y}{2}, \g(x,y)))b(y)dyd\x.
\end{align*}
Thus it suffices to compute $L$.
Since $V$ is a polynomial of degree $2$ with respect to $\x$-varibble, we have
\begin{align*}
V(\frac{x+y}{2},\x+\g(x,y) )=&V(\frac{x+y}{2}, \g(x,y))+\x\cdot \pa_{\x}V(\frac{x+y}{2},\g(x,y))\\
&+\frac{1}{2}\x\cdot \pa_{\x}^2V(\frac{x+y}{2},\g(x,y)) \cdot\x.
\end{align*}
Note that $\pa_{\x}^2V(\frac{x+y}{2},\g(x,y))=\pa_{\x}^2V(\frac{x+y}{2},0)$ since $V$ is a second order differential operator. By integrating by parts, $L(x)$ is written as
\begin{align*}
&\frac{1}{(2\pi)^n}\int_{\re^{2n}}e^{i(x-y)\cdot \x}(\x\cdot \pa_{\x}V(\frac{x+y}{2},\g(x,y))+ \x\cdot \pa_{\x}^2V(\frac{x+y}{2},\g(x,y)) \cdot\x) b(y)dyd\x\\
&=D_y(\pa_{\x}V(x,\g(x,y)) b(x))|_{x=y}+(D_y^2(\pa_{\x}^2V(\frac{x+y}{2},0) b(y))|_{x=y}\in S^{k+\a-\m-1}(\re^n).
\end{align*}
This completes the proof.

\end{proof}

Now we find approximate solutions to the eikonal equations:
\begin{align}\label{eikonalr}
(\nabla\f(x))^2-|x|^{2\a}+V(x,\nabla\f(x))-z=0.
\end{align}
In \cite{GY}, solutions to eikonal equations is used for constructing eigenfunctions of a usual Schr\"odinger operator $-\Delta+V$ with a long range perturbation. Isozaki \cite{I1} proved the existence of solutions to eikonal equations for $-\Delta+V$ by using the estimates for the classical trajectories. In our case, we cannot directly apply this strategy since the classical trajectories may blow up at finite time. Instead, we use iteration and construct the approximate solutions to $(\ref{eikonalr})$ even for $z\notin \re$.

\begin{prop}\label{eikonal}
Set $\f_{0,\pm}(x)=\f_{0,\pm}(x,z)=\pm\frac{|x|^{\a+1}}{1+\a}\pm z\frac{|x|^{1-\a}}{2(1-\a)}$. Let $R\geq 1$ be large enough. Then for any integer $N>0$, there exists $\f_{N,\pm}\in S^{1+\a}(\re^n)$ such that $\f_{N,\pm}-\f_{N-1,\pm}\in S^{1+\a-N\m}(\re^n)$, $\Im (\f_{N,\pm}-\f_{0,\pm})\in S^{0}(\re^n)$, $\f_{N,\pm}-\f_{N-1,\pm}$ is supported in $|x|\geq R$ and
\begin{align}\label{eikonalapp}
(\nabla\f_{N,\pm}(x))^2-|x|^{2\a}+V(x,\nabla\f_{N,\pm}(x))-z\in S^{2\a-(N+1)\m}(\re^n).
\end{align}
\end{prop}

\begin{rem}
Such construction of $\f_N$ succeeds for $0<\a< 1$. For $\a=1$ and $z\in \re$, we have to replace $\f_{0,\pm}(x,z)=\pm\frac{|x|^{2}}{2}\pm \frac{z}{2}\log|x|$.
\end{rem}

\begin{proof}
We find $\f_{N,\pm}\in S^{1+\a}(\re^n)$ of the form
\begin{align*}
\f_{N,\pm}(x)=\f_{0,\pm}(x)+\sum_{j=1}^Ne_{j,\pm}(x),\,\, e_{j,\pm}\in S^{1+\a-j\m}.
\end{align*}
By a simple calculation, we have
\begin{align*}
&(\nabla\f_{N,\pm}(x))^2-|x|^{2\a}+V(x,\nabla\f_{N,\pm}(x))-z\\
=&\frac{z^2}{4}|x|^{-2\a}+2\sum_{j=1}^N\nabla\f_{0,\pm}\cdot \nabla e_{j,\pm}+\sum_{j,k=1}^N\nabla e_{j,\pm}\cdot \nabla e_{k,\pm}+V(x,\nabla\f_{N,\pm}(x)).
\end{align*}

We set 
\begin{align*}
e_{1,\pm}(x)=&\mp\int_{\frac{R}{2}}^{|x|} \frac{1}{2s^{\a}}(V(s\hat{x},\nabla {\f}_{0,\pm}(s\hat{x}))-\frac{z^2}{4}s^{-2\a} )ds\bar{\chi}_R(x)\in S^{\a+1-\m}(\re^n)\\
\f_{1,\pm}(x)=&\f_{0,\pm}(x)+e_{1,\pm}(x).
\end{align*}
Note $\Im e_{1,\pm}\in S^{1-\a-\m}(\re^n)$.
Then $(\nabla \f_{1,\pm}(x))^2-|x|^{2\a}+V(x,\nabla\f_{1,\pm}(x))-z$ is equal to
\begin{align*}
&(\Im \nabla\f_{0,\pm})\cdot \nabla e_{1,\pm}+ (\nabla e_{1,\pm})^2+V(x,\nabla\f_{1,\pm})-V(x,\nabla \f_{0,\pm})\\
&=(\Im \nabla\f_{0,\pm})\cdot \nabla e_{1,\pm}+(\nabla e_{1,\pm})^2+\int_0^1\nabla e_{1,\pm}\cdot(\pa_{\x}V)(x,\nabla\f_{0,\pm}+t \nabla e_{1,\pm})dt,
\end{align*}
and this term belongs to $S^{2\a-2\m}(\re^n)$.
In fact, $\nabla\f_{0,\pm}+t \nabla e_{1,\pm}(x)=|x|^{\a-1}x+O(|x|^{\a-\m})$ and hence $\pa_{\x}V(x,\nabla\f_{0,\pm}+t \nabla e_{1,\pm})=O(|x|^{\a-\m})$ uniformly in $0\leq t\leq 1$. 

For $N\geq 1$, we define $\f_N\in S^{\a+1}$ and $e_{N}\in S^{\a+1-N\m}$ inductively as follows:
\begin{align*}
\f_{N+1,\pm}(x)=&\f_{N,\pm}(x)+e_{N+1,\pm}(x),\,e_{N+1,\pm}(x)=\mp\int_{\frac{R}{2}}^{|x|}\frac{E_{N+1}(s\hat{x})}{2s^{\a}}ds\bar{\chi}_R(|x|),\\
E_{N+1,\pm}=&\sum_{\substack{j+k=N+1,\\1\leq j,k\leq N}}\nabla e_{j,\pm}\cdot \nabla e_{k,\pm}+V(x, \nabla\f_{N,\pm})-V(x,\nabla\f_{N-1,\pm})\\
&-2(\Im \nabla \f_{0,\pm})\cdot \nabla e_{N,\pm}.
\end{align*}
We note $\Im e_{N,\pm}\in S^{1-\a-N\m}(\re^n)$. 
For $|x|\geq 2R$, we have
\begin{align*}
(\nabla\f_{N+1,\pm})^2-|x|^{2\a}-z=&(\nabla\f_{0,\pm}(x)+\sum_{j=1}^{N+1}\nabla e_{j,\pm})^2-|x|^{2\a}-z\\
\equiv&2\sum_{j=1}^{N+1}\nabla\f_{0,\pm}\cdot \nabla e_{j,\pm}+\sum_{m=2}^{N+1}\sum_{j+k=m}\nabla e_{j,\pm}\cdot\nabla e_{k,\pm}\\
=&-V(x, \nabla\f_{N,\pm}(x))
\end{align*}
modulo $S^{2\a-(N+2)\m}$. Hence 
\begin{align*}
|\nabla\f_{N+1,\pm}|^2-|x|^{2\a}+V(x,\nabla\f_{N+1}(x))\equiv& V(x,\nabla\f_{N+1}(x))-V(x,\nabla\f_{N}(x))\\
\equiv& 0
\end{align*}
modulo $S^{2\a-(N+2)\m}$. Moreover, we have $\Im(\f_{N,\pm}\mp z\frac{|x|^{1-\a}}{2(1-\a)})\in S^0(\re^n)$ since $\Im e_{N,\pm}\in S^{1-\a-N\m}(\re^n)$ and $\a>1$. This completes the proof.
\end{proof}

\begin{proof}[Proof of Theorem \ref{parathm}]
Fix a signature $\pm$. Let $N>0$ be an integer such that 
\begin{align*}
2\a-(N+1)\m<-\frac{n+1-\a}{2}-\m.
\end{align*}
We take $\f=\f_{\pm}=\f_{\pm,N}$ as in Proposition \ref{eikonal}. Then Proposition \ref{paraprop} gives Theorem \ref{parathm}.
\end{proof}

\subsection{Construction of the $L^2$-solutions, proof of Theorem \ref{thm2}}
Now we construct the $L^2$-solutions to
\begin{align*}
(P-z)u=0,
\end{align*}
where $u$ is of the form 
\begin{align}\label{efform}
u(x)=u_0(x)+u_1(x),\,\, u_0(x)=e^{i\f_-(x)}b(x),\,\, u_1\in L^{2, \frac{\a-1}{2}+tm(x,\x)}.
\end{align}

\begin{proof}[Proof of Theorem \ref{thm2}]

Set $\tilde{V}(x,\x)=V(x,\x)-(\jap{x}^{2\a}-|x|^{2\a})\bar{\chi}(2|x|/R)$ for $R>0$. Let $\f_{-}\in S^{1+\a}$ and $b=|x|^{-\frac{n-1+\a}{2}}\bar{\chi}(|x|/R)a(\hat{x})$ be as in Theorem \ref{parathm} with $\tilde{V}$, where $a\in C^{\infty}(\mathbb{S}^{n-1})\setminus \{0\}$. Since $\bar{\chi}(2|x|/R)\bar{\chi}(|x|/R)=\bar{\chi}(|x|/R)$ and $S^{-\frac{n+1-\a}{2}-\m}(\re^n)\subset L^{2, \frac{1-\a+\m}{2}}$, we have
\begin{align}\label{u_0}
(P-z)(e^{i\f_{-}}b)\in L^{2, \frac{1-\a+\m}{2}}.
\end{align}

Now we take $0<t<\min(\m/2, (\a-1)/2)$ and $m=m_{R_0}$ be as in subsection \ref{subsecesc}, where $R_0$ is as in Lemma \ref{escape}. Since 
\begin{align*}
L^{2,(1-\a)/2+tm(x,\x)}\subset L^{2, \frac{1-\a+\m}{2}},\,\,z\in \mathbb{C}\setminus T_{\a,t},
\end{align*}
then there exists $u_1\in L^{2,(\a-1)/2+tm(x,\x)}$ such that
\begin{align*}
(P-z)u_1=-(P-z)(e^{i\f_{-}}b).
\end{align*}
by Theorem \ref{thm1}. We set $u=u_1+e^{i\f_{-}}b\in L^2$, then $u$ satisfies $(P-z)u=0$ since $t<(\a-1)/2$. Finally, we prove $u\neq 0$. In order to prove this, we use the wavefront condition of $u_1$ and $e^{i\f_{-}}b$.

\begin{lem}\label{WF1}
Set $u_0=e^{i\f_{-}}b$, where $b(x)=|x|^{-\frac{n-1+\a}{2}}\bar{\chi}(|x|/R)a(\hat{x})$ and $a\in C^{\infty}(\mathbb{S}^{n-1})\setminus \{0\}$. Let $b_{R_1,\d}(x,\x)=\chi((\y(x,\x)+1)/\d)a_{R_1}(x,\x)$ and $A_{R_1,\d}=\Op(b_{R_1,\d})$ for $0<\d<1$ small enough and $R_1\geq 1$ large enough. Then $A_{R_1,\d}u_0\notin L^{2,\frac{\a-1}{2}}$.

\end{lem}

\begin{proof}
By $(\ref{u_0})$, Proposition \ref{elliptic} implies that $(1-\Op(a_R))u_0\in L^{2,(\a-1)/2}$.
Moreover, by a simple calculation, we have
\begin{align}\label{rad}
|x|^{-\a-1}(x\cdot D_x- x\cdot \pa_x\f_-(x))u_0\in \bigcap_{\e>0} L^{2,\frac{\a-1}{2}+1-\e}\subset L^{2,\frac{\a-1}{2}}.
\end{align}
Note that if $r_1,\d$ are small and $R_1$ is large, for $(x,\x)\in \supp (a_{R_1}-b_{R_1,\d})$
\begin{align*}
|x\cdot \x-x\cdot \pa_{x}\f_-(x)|\geq C|x|^{1+\a}.
\end{align*}
Since $u_0\notin L^{2,(\a-1)/2}$ and $u_0\in \cap_{\e>0}L^{2,(\a-1)/2-\e}$, we have
\begin{align*}
\Op(a_{R_1}-b_{R_1,\d})u_0=&\Op(\frac{a_{R_1}-b_{R_1,\d}}{x\cdot \x-x\cdot \pa_x\f_-(x)}|x|^{1+\a} )\\
&\cdot |x|^{-1-\a}(x\cdot D_x-x\cdot \pa_x\f_-(x))u_0+L^{2,\frac{\a-1}{2}} \in L^{2,\frac{\a-1}{2}}.
\end{align*}
by a symbol calculus and $(\ref{rad})$. Thus if we suppose $A_{R_1,\d}u_0\in L^{2,\frac{\a-1}{2}}$, then $u_0\in L^{2,\frac{\a-1}{2}}$ follows. However, this is a contradiction since $u_0\notin L^{2,(\a-1)/2}$ by a simple calculation.
\end{proof}

\begin{lem}\label{WF2}
For $0<\d<1$ small enough and $R_1\geq 1$ large enough, $A_{R_1,\d}u_1\in L^{2,\frac{\a-1}{2}}$.
\end{lem}

\begin{proof}
Note that $u_1\in L^{2,(\a-1)/2-tm(x,\x)}=\Op(\tilde{G}_{(\a-1)/2, -tm})^{-1}L^2$, $0<t<(\a-1)/2$ and $\tilde{G}_{(\a-1)/2,-tm}=\jap{x}^{(\a-1)/2-tm(x,\x)}$ by $(\ref{Ginvdef})$. Moreover, we note $m(x,\x)=-1$ on $\supp b_{R_1,\d}$ if $0<\d<1$ is small enough and $R_1\geq 1$ is large enough. Thus $A_{R_1,\d}u\in L^{2,(\a-1)/2}$.
\end{proof}

By the above two lemmas, we obtain $u=u_0+u_1\neq 0$. This completes the proof of Theorem \ref{thm2}.

\end{proof}

Finally, we prove that there are many eigenfunctions associated with $\l\in \mathbb{C}\setminus T_{\a,t}$.
\begin{prop}\label{many}
Suppose that $a,a'\in C^{\infty}(\mathbb{S}^{n-1})$ are linearly independent. Let $u,u' \in L^2\setminus$ be corresponding eigenfunctions as in $(\ref{efform})$. Then $u,u'$ are also linearly independent.
\end{prop}
\begin{proof}
By $(\ref{efform})$ and Lemma \ref{WF2}, we write
\begin{align*}
u(x)=&e^{i\f_{-}(x)}|x|^{-\frac{n-1+\a}{2}}\bar{\chi}(|x|/R)a(\hat{x})+u_1(x),\\
u'(x)=&e^{i\f_{-}(x)}|x|^{-\frac{n-1+\a}{2}}\bar{\chi}(|x|/R)a'(\hat{x})+u_1'(x),
\end{align*}
where $u_1,u_1'\in L^2$ satisfy $A_{R_1,\d}u_1, A_{R_1,\d}u_1'\in L^{2,\frac{\a-1}{2}}$, where $A_{R_1,\d}$ is defined in Lemma \ref{WF1}. Suppose that $L, L'\in \mathbb{C}$ satisfy
\begin{align}\label{linindep}
Lu(x)+L'u'(x)=0,\,\, x\in \re^n.
\end{align}
It suffices to prove that $La(\hat{x})+L'a'(\hat{x})= 0$ for $\hat{x}\in S^{n-1}$. Suppose $La(\hat{x})+L'a'(\hat{x})\neq 0$ for some $\hat{x}\in S^{n-1}$. By Lemma \ref{WF1}, we have
\begin{align}\label{lin2}
A_{R_1,\d}(e^{i\f_{-}(x)}|x|^{-\frac{n-1+\a}{2}}\bar{\chi}(|x|/R)(La(\hat{x})+L'a'(\hat{x}) ))\notin L^{2,\frac{\a-1}{2}}.
\end{align}
$(\ref{linindep})$ and $(\ref{lin2})$ imply
\begin{align*}
A_{R_1,\d}(Lu+L'u')\notin L^{2,\frac{\a-1}{2}}.
\end{align*}
This is a contradiction.
\end{proof}

\section{Proof of Theorem \ref{thm3} and Corollary \ref{disc1}}

\subsection{Proof of Theorem \ref{thm3}}.

Note that if $\a\leq 1$, then $D^{\a}_{\mathrm{min}}=\{u\in L^2\mid Pu\in L^2\}$ since $P$ is essentially self-adjoint on $\mathcal{S}(\re^n)$ for $\a\leq 1$. However, it follows that $D^{\a}_{\mathrm{min}}\neq \{u\in L^2\mid Pu\in L^2\}$ for $\a>1$. 

\begin{lem}
Let $\a>1$. For $\d>0$, there exists $C>0$ such that
\begin{align}\label{smo1}
\|\Op(a_{2R})u\|_{L^{2,\frac{\a-1-\d}{2}}}\leq  C\|Pu\|_{L^2}+C\|u\|_{L^2}
\end{align}
for $u\in D^{\a}_{\mathrm{min}}$, where we recall that $a_{2R}$ is as in $(\ref{cut2})$.
\end{lem}

\begin{proof}
First, we prove $(\ref{smo1})$ for $u\in \mathcal{S}(\re^n)$.
We may assume $0<\d<\m$. Set
\begin{align*}
b_{R}(x,\x)=a_{2R}(x,\x)^2\frac{x\cdot\x}{|x||\x|}\int_1^{|x|/R}s^{-1-\d}ds\in S^{0,0}.
\end{align*}
We note that $|x|>2R$, $|\x|\geq 2R$ and $|x|^{\a}\sim |\x|$ hold for $(x,\x)\in \supp b_{R}$.
For $(x,\x)\in \supp b_{R}$, we have
\begin{align*}
H_{p_0}(\frac{x\cdot\x}{|x||\x|}\int_1^{|x|/R}s^{-1-\d}ds)=&2\frac{|x|^2|\x|^2-(x\cdot\x)^2}{|x||\x|}(\frac{1}{|x|^2}+\a\frac{|x|^{2\a-2}}{|\x|^2})\int_{1}^{\frac{|x|}{R}}\frac{1}{s^{1+\d}}ds\\
&+2R^{\d}\frac{(x\cdot\x)^2}{|x|^{3+\d}|\x|}\\
\geq&C\jap{x}^{\a-1-\d}
\end{align*}
with $C>0$ if $R>0$ is large enough. Since $H_{V}b_{R}\in S^{0,\a-1-\m}$ and $0<\d<\m$, we see
\begin{align*}
H_pb_{R}\geq C\jap{x}^{\a-1-\d}a_{2R}^2+e_{R},
\end{align*}
where $e_{R}\in S^{0,\a-1}$ is supported away from the elliptic set of $P$. By the sharp G$\rm{\mathring{a}}$rding inequality, we have
\begin{align}\label{smoes}
(u,[P, i\Op(b_{R})]u)_{L^2}\geq \|\Op(a_{2R})u\|_{L^{2,\frac{\a-1-\d}{2}}}^2+(u, \Op(e_{R})u)_{L^2}-C\|u\|_{H^{-\frac{1}{2},\frac{\a}{2}-1}}^2
\end{align}
for any $u\in \mathcal{S}(\re^n)$. Take $R_1\geq 1$ such that $a_{2R}a_{R_1}=a_{2R}$. Substituting $\Op(a_{R_1})$ into $(\ref{smoes})$ and using the disjoint support property and a support property of $a_{R_1}$, then we have
\begin{align*}
(u,[P, i\Op(b_{R})]u)_{L^2}\geq \|\Op(a_{2R})u\|_{L^{2,\frac{\a-1-\d}{2}}}^2+(u, \Op(e_{R})u)_{L^2}-C\|u\|_{L^2}^2
\end{align*}
for $u\in \mathcal{S}(\re^n)$ with some $C>0$. Using the elliptic estimate Proposition \ref{elliptic} in order to estimate the term $(u, \Op(e_{R})u)_{L^2}$, we have
\begin{align*}
\|\Op(a_{2R})u\|_{L^{2,\frac{\a-1-\d}{2}}}\leq  C\|Pu\|_{L^2}+C\|u\|_{L^2}
\end{align*}
for $u\in \mathcal{S}(\re^n)$ with some $C>0$. Thus we obtain $(\ref{smo1})$ for $u\in \mathcal{S}(\re^n)$. 

In order to prove $(\ref{smo1})$ for $u\in D^{\a}_{\mathrm{min}}$, it remains to use a standard density argument. Let $u\in D^{\a}_{\mathrm{min}}$. By definition of $D^{\a}_{\mathrm{min}}$, there exists $u_k\in C_c^{\infty}(\re^n)$ such that $u_k\to u$ and $Pu_k\to Pu$ in $L^2(\re^n)$. Substituting $u_k$ into $(\ref{smo1})$, we have
\begin{align*}
\sup_{k}\|\Op(a_{2R})u_k\|_{L^{2,\frac{\a-1-\d}{2}}}<\infty.
\end{align*}
Hence $\Op(a_{2R})u_k$ has a weak*-convergence subsequence in $L^{2,\frac{\a-1-\d}{2}}$ and its accumulation point is $\Op(a_{2R})u$. Thus we obtain $\Op(a_{2R})u\in L^{2,\frac{\a-1-\d}{2}}$ and 
\begin{align*}
\|\Op(a_{2R})u\|_{L^{2,\frac{\a-1-\d}{2}}}\leq \liminf_{k\to \infty}\|\Op(a_{2R})u_k\|_{L^{2,\frac{\a-1-\d}{2}}}\leq C\|Pu\|_{L^2}+C\|u\|_{L^2}.
\end{align*}

\end{proof}

Combining this lemma with the elliptic estimate Proposition \ref{elliptic} and Lemma \ref{suppell}, we have the following proposition:

\begin{prop}
Let $\a>1$ and $0\leq \b_1,\b_2\leq 4$ with $\b_1+\b_2=1$. For $\d>0$, there exists $C>0$ such that
\begin{align}\label{smo2}
\|u\|_{H^{\frac{\a-1-\d}{2\a}\b_1, \frac{\a-1-\d}{2}\b_2}}\leq  C\|Pu\|_{L^2}+C\|u\|_{L^2}
\end{align}
for $u\in D^{\a}_{\mathrm{min}}$. In particular, the natural embedding $D^{\a}_{\mathrm{min}}\hookrightarrow L^2(\re^n)$ is compact, where we regard $D^{\a}_{\mathrm{min}}$ as a Banach space equipped with its graph norm.
\end{prop}

This proposition gives the proof of Theorem \ref{thm3}.

\subsection{Proof of Corollary \ref{disc1}}

Note that $D_{min}^{\a}$ is the domain of the closure of $P|_{C_c^{\infty}(\re)}$. Set 
\begin{align*}
D^{\a}=\{u\in L^2(\re^n)\mid Pu\in L^2(\re^n)\}.
\end{align*}
We easily see that $D^{\a}$ is the domain of $(P|_{C_c^{\infty}(\re)})^*$. Moreover, it follows that the action of $(P|_{C_c^{\infty}(\re)})^*$ on $D^{\a}$ is in the distributional sense. In particular, we have
\begin{align*}
\Ker((P|_{C_c^{\infty}(\re)})^*\mp i)=\Ker_{L^2}(P\mp i).
\end{align*}

We use the following von-Neumann theorem.

\begin{lem}\cite[Theorem X.2 and Corollary after Theorem X.2]{RS}\label{von}
Set $\mathcal{H}_{\pm}=\Ker_{L^2}(P\mp i)$. Then there is a one-to-one correspondence between self-adjoint extensions of $P|_{C_c^{\infty}(\re)}$ and unitary operators from $\mathcal{H}_+$ to $\mathcal{H}_-$. Moreover, for $U\in B(\mathcal{H}_+,\mathcal{H}_-)$ be a unitary operator , we define
\begin{align*}
D_U=\{v+w+Uw\mid v\in D_{min}^{\a}, w\in \mathcal{H}_+\}.
\end{align*}
Then $P$ is self-adjoint on $D_U$.
\end{lem}

Now suppose $n=1$. We prove that each self-adjoint extension of $P|_{C_c^{\infty}(\re)}$ has a discrete spectrum.

\begin{lem}\label{one2}
$\dim\mathcal{H}_{+}=\dim\mathcal{H}_{-}=2$.
\end{lem}
\begin{proof}
By \cite[Theorem X.1]{RS}, it suffices to prove that 
\begin{align*}
\dim\Ker_{L^2}(P-i\m)=\dim\Ker_{L^2}(P+i\m)=2
\end{align*}
for some $\m>0$.
We note $\dim\Ker_{L^2}(P\pm i\m)\leq 2$ by uniqueness of solutions to ODE. Hence it suffices to prove $\dim\Ker_{L^2}(P\pm i\m)\geq 2$. We observe $\mathbb{S}^{n-1}=\mathbb{S}^0=\{\pm 1\}$ and $\dim C^{\infty}(\{\pm 1\})=2$. By Proposition \ref{many}, the discreteness of $T_{\a,t}$ imply that for some $\m\in \mathbb{C}\setminus \re\cup T_{\a,t}$ there exists linearly independent functions such that $u_{\pm}, u'_{\pm}\in \Ker_{L^2}(P\pm i\m)$. This gives $\dim\Ker_{L^2}(P\pm i\m)\geq 2$.
\end{proof}

The following proposition is a variant of \cite[Theorem XIII.64]{RS}. We do not know whether a self-adjoint extension of $P|_{C_c^{\infty}(\re^n)}$ is bounded from below. Hence we cannot apply \cite[Theorem XIII.64]{RS} with our case directly in order to prove Corollary \ref{disc1}.

\begin{prop}\label{cptcrprop}
Let $\mathcal{H}$ be a Hilbert space and $A$ be a self-adjoint operator on $\mathcal{H}$. Suppose that $(A+i)^{-1}$ is a compact operator on $\mathcal{H}$. Then there exists $\{\l_j\}_{j=1}^{\infty}\subset \re$ such that $|\l_k|\to \infty$ as $k\to\infty$ and $\s(A)=\s_d(A)=\{\l_j\}_{j=1}^{\infty}$, where $\s(A)$ is the spectrum of $A$ and $\s_d(A)$ is the discrete spectrum of $A$. 
\end{prop}
\begin{proof}
First, we prove existence of $\l\in \re\setminus \s(A)$. To prove this, we use a contradiction argument. Suppose $\s(A)=\re$.
Set $B=(A-i)^{-1}(A+i)^{-1}=f(A)$, where $f(t)=1/(t^2+1)$. By the spectrum mapping theorem, we have $\s(B)=[0,1]$. On the other hand, by the assumption of the lemma, it follows that $B$ is a compact self-adjoint operator on $\mathcal{H}$. This contradicts to $\s(B)=[0,1]$.

We let $\l\in \re\setminus \s(A)$ and set $T=(A-\l)^{-1}$. Since $(A+i)^{-1}$ is compact and since $\l\in \re$, it easily follows that $T$ is a compact self-adjoint operator. By the Hilbert-Schmidt theorem \cite[Theorem VI.16]{RS}, there exist a complete orthonormal basis $\f_k\in \mathcal{H}$ and a sequence $\m_k\in\re$ such that 
\begin{align}\label{cptef}
T\f_k=\m_k\f_k,\,\, \m_k\to 0\,\, \text{as}\,\, k\to \infty.
\end{align}
We note that $\f_k$ belongs to the domain of $A$ since $\f_k\in \Ran T$ and since $\Ran T$ is contained in the domain of $A$. Moreover, we observe $\m_k\neq 0$. In fact, suppose $\m_k=0$ holds. Multiplying $(\ref{cptef})$ by $A-\l$, we have $\f_k=0$, which is a contradiction. By $(\ref{cptef})$, we have
\begin{align*}
A\f_k=\l_k\f_k,\,\, \l_k=\l+1/\m_k.
\end{align*}
Note $|\l_k|\to\infty$ as $k\to \infty$.
Since $\l_k$ has no accumulation point in $\re$, it suffices to prove $\s(A)=\{\l_k\}_{k=1}^{\infty}$. To see this, we prove that $A-z$ has a bounded inverse for $z\in \re\setminus \{\l_k\}_{k=1}^{\infty}$. We set 
\begin{align}\label{resol}
R(z)\g=\sum_{k=1}^{\infty}\frac{1}{\l_k-z}(\f_k, \g)\f_k,\,\, \g\in \mathcal{H}
\end{align}
and $c=\inf_{k\geq 1}|\l_k-z|$. Since $\l_k$ has no accumulation point in $\re$, we have $c>0$. Thus we have
\begin{align*}
\sum_{k=1}^{\infty}\frac{1}{|\l_k-z|^2}|(\f_k,\g)|^2\leq c^{-2}\sum_{k=1}^{\infty}|(\f_k,\g)|^2.
\end{align*}
Hence $R(z)$ is a bounded operator on $\mathcal{H}$. Moreover, $(A-z)R(z)\g=\g$ holds by $(\ref{resol})$. These imply $z\notin \re\setminus\s(A)$. Thus we have $\s(A)=\{\l_j\}_{j=1}^{\infty}$.
Moreover, it follows that $\s_d(A)=\s(A)$ holds since $\dim\Ker(A-\l_k)=\dim\Ker(T-\m_k) <\infty$.
\end{proof}

By virtue of Lemma \ref{one2} and \cite[Corollary after Theorem X.2]{RS}, it follows that $P|_{C_c^{\infty}(\re)}$ has a self-adjoint extension.

\begin{proof}[Proof of Corollary \ref{disc1}]
Fix $U\in $ be a unitary operator and let $D_U$ be as in Lemma \ref{von}.
By virtue of Proposition \ref{cptcrprop}, it suffices to prove that the inclusion $D_{U}\subset L^2$ is compact, where we regard $D_U$ as a Hilbert space equipped with the graph norm of $P$. 
Let $\f_j\in D_{U}$ be a bounded sequence in $D_{U}$:
\begin{align*}
\sup_{j}(\|\f_j\|_{L^2}+\|P\f_j\|_{L^2}) <\infty.
\end{align*}
We only need to prove that $\f_j$ has a convergent subsequence in $L^2$.
We write $\f_j=u_j+v_j+Uv_j$, where $u_j\in D_{min}^{\a}$ and $v_j\in \mathcal{H}_+$.
By \cite[Lemma  before Theorem X.2]{RS}, we see that
\begin{align*}
0=(u_j,v_j)_{L^2}+(Pu_j,Pv_j)_{L^2}=&(v_j,Uv_j)_{L^2}+(Pv_j,PUv_j)_{L^2}\\
=&(u_j,Uv_j)_{L^2}+(Pu_j,PUv_j)_{L^2}.
\end{align*}
Therefore, $u_j$ and $v_j$ are bounded in $D_U$. Since $u_j\in D_{min}^{\a}$, it follows that $u_j$ has a convergent subsequence $\{u_{j_k}\}$ in $L^2$. Moreover, we see that $v_{j_k}\in \mathcal{H}_+$ has a convergent subsequence in $L^2$ due to the finiteness of the dimension of $\mathcal{H}_+$. Thus we conclude that $\f_j$ has a convergent subsequence in $L^2$.
\end{proof}

\appendix

\section{Variable order spaces}\label{appB}

In this Appendix, we give a construction of variable order weighted $L^2$-spaces. Here, we follow the argument in \cite{FRS}. See \cite[Appendix A]{BVW} for other ways of constructions.

Let $m\in S^{0,0}$ be real-valued and $k,t\in \re$. Suppose $|m(x,\x)|\leq 1$ for $(x,\x)\in \re^{2n}$. Set $G_{k,tm}(x,\x)=\jap{x}^{k+tm(x,\x)}$. Set $l(x)=\jap{\log\jap{x}}$.

\begin{defn}\label{def1app}
 For $a\in C^{\infty}(\re^{2n})$, we say that for  $a\in S^{s,k+tm(x,\x)}$ if
\begin{align*}
|\pa_{x}^{\c_1}\pa_{\x}^{c_2}a(x,\x)|\leq C_{\c_1\c_2}l(x)^{|\c_1|+|\c_2|} \jap{x}^{k+tm(x,\x)-|\c_1|}\jap{\x}^{s-|\c_2|}
\end{align*}
for $\c_1,\c_2\in \mathbb{N}^n$. 
\end{defn}
Note that $G_{k,tm} \in S^{0,k+tm(x,\x)}$.

\begin{lem}
An unbounded operator $\Op(G_{k,tm})$ on $L^2(\re^n)$ with domain $\mathcal{S}(\re^n)$ admits a self-adjoint extension.
\end{lem}

\begin{proof}
By virtue of\cite[Theorem X.23]{RS}, it suffices to prove that $\Op(G_{k,tm})$ is bounded below in $\mathcal{S}(\re^n)$.
We note $G_{k,tm}(x,\x)=G_{k/2,tm/2}(x,\x)^2$. By the standard construction (see \cite[Lemma 13]{FRS}), there exists $R_j\in S^{-j,k/2-j+0+tm(x,\x)/2}$ such that
\begin{align*}
(&G_{k/2,tm/2}(x,\x)^2+\sum_{j=1}^NR_j)^*(G_{k/2,tm/2}(x,\x)^2+\sum_{j=1}^NR_j)\\
&\in S^{-(N+1), k-(N+1)+0+tm(x,\x)},
\end{align*}
where $(\cdot )^*$ denotes the adjoint symbol. By the Borel summation theorem, we have
\begin{align*}
G_{k,tm}(x,\x)=b^*b+e,\,\, b\in S^{k/2+tm(x,\x)/2},\,\, e\in S^{-\infty,-\infty}.
\end{align*}
Thus we obtain
\begin{align*}
(u, \Op(G_{k,tm})u)\geq (u,\Op(e)u)\geq -C\|u\|_{L^2}^2,\,\, u\in \mathcal{S}(\re^n).
\end{align*}
\end{proof}

We denote a self-adjoint extension of $\Op(G_{k,tm})$ in $L^2(\re^n)$ by $G(t)$ and its domain by $D_{G(t)}$.  


\begin{lem}\label{varinv}
There exists $R_1(t)\in \Op S^{-\infty, -\infty}$ such that $\Op(G_{k,tm})+R_1(t)$ is invertible in $\mathcal{S}(\re^n)\to \mathcal{S}(\re^n)$. Moreover, its inverse is a pseudodifferential operator with its symbol in $S^{0,-k-tm(x,\x)}$. Moreover, the symbol of its inverse is $G_{-k,-tm}+S^{-1,-k-1-tm(x,\x)+0}$.
\end{lem}

\begin{proof}
We follow the argument as in \cite[Appendix Lemma 12]{FRS}. We decompose $L^2=\overline{\Ran_{L^2} G(t)}\oplus \Ker_{L^2}G(t)$. We denote the orthogonal projection into $\Ker_{L^2}G(t)$ by $\pi(t): L^2\to \Ker_{L^2}G(t)$. By the standard parametrix construction of $G(t)$, we see that $\Ker_{L^2}G(t)\subset \mathcal{S}(\re^n)$ and $\Ker_{L^2}G(t)$ is of finite dimension. This implies $\pi(t)\in \Op S^{-\infty,-\infty}$. We define $\tilde{G}(t)=G(t)(I-\pi(t))+\pi(t)\in \Op S^{0,k+tm(x,\x)}$. We observe that $\tilde{G}(t):D_{G(t)}\to L^2$ is invertible. We set $R_1(t)=(I-G(t))\pi(t)\in \Op S^{-\infty,-\infty}$, then $\tilde{G}(t)=G(t)+R_1(t)$.

We show that $\tilde{G}(t)$ is invertible in $\mathcal{S}(\re^n)\to \mathcal{S}(\re^n)$. This map is injective since $\tilde{G}(t)$ is injective in $D_{G(t)}\to L^2$. Next, we prove that $\tilde{G}(t):\mathcal{S}(\re^n)\to \mathcal{S}(\re^n)$ is surjective. To see this, let $f\in \mathcal{S}(\re^n)$. Since $\tilde{G}(t):D_{G(t)}\to L^2$ is invertible, there exists $u\in D_{G(t)}$ such that $G(t)u=f$. By using existence of the parametrix of $\tilde{G}(t)$, we obtain $u\in \mathcal{S}(\re^n)$.

Finally, we show that the inverse of $\tilde{G}(t)$ belongs to $\Op S^{0,-k-tm(x,\x)}$ and its symbol is $G_{-k,-tm}+S^{-1,-k-1-tm(x,\x)+0}$. Let $Q(t)$ is the parametrix of $\tilde{G}(t)$: $Q(t)\tilde{G}(t)=I+R_2(t)$, where $R_2(t)\in \Op S^{-\infty,-\infty}$. Then the symbol of $Q(t)$ is $G_{-k,-tm}+S^{-1,-k-1-tm(x,\x)+0}$. Moreover, we observe 
\begin{align*}
Q(t)=Q(t)\tilde{G}(t)\tilde{G}(t)^{-1}=\tilde{G}(t)^{-1}+R_2(t)\tilde{G}(t)^{-1}
\end{align*}
in $\mathcal{S}(\re^n)\to \mathcal{S}(\re^n)$. By the open mapping theorem, $\tilde{G}(t)^{-1}$ is continuous in $\mathcal{S}(\re^n)\to \mathcal{S}(\re^n)$. Thus we have $R_2(t)\tilde{G}(t)^{-1}\in \Op S^{-\infty,-\infty}$. We conclude that $\tilde{G}(t)=Q(t)-R_2(t)\tilde{G}(t)\in \Op S^{0,-k-tm(x,\x)}$ and its symbol is $G_{-k,-tm(x,\x)}+S^{-1,-k-1-tm(x,\x)+0}$.

\end{proof}

Let $\tilde{G}_{k,tm}\in S^{0,k+tm(x,\x)}$ such that 
\begin{align}\label{Ginvdef}
\Op(\tilde{G}_{k,tm})=\Op(G_{k,tm})+R_1(t).
\end{align}
By Lemma \ref{varinv} and duality, $\Op(\tilde{G}_{k,tm}):\mathcal{S}'(\re^n)\to \mathcal{S}'(\re^n)$ is also invertible.

Now we define the variable order weighted $L^2$ space by
\begin{align}\label{vardef}
L^{2,k+tm(x,\x)}=\{u\in \mathcal{S}'(\re^n)\mid \Op(\tilde{G}_{k,tm})u\in L^2(\re^n) \}
\end{align}
for $k\in \re$ and $|t|<1/2$ and its inner metric by
\begin{align*}
(u,v)_{L^{2,k+tm(x,\x)}}=(\Op(\tilde{G}_{k,tm})u,\Op(\tilde{G}_{k,tm})v)_{L^2}.
\end{align*}
Then $L^{2,k+tm(x,\x)}$ is a Hilbert space.

We state some properties of $L^{2,k+tm(x,\x)}$.
\begin{lem}\label{varequi}
\begin{itemize}
\item[$(i)$] $(L^{2,k+tm(x,\x)})^*=L^{2,-k-tm(x,\x)}$.\\
\item[$(ii)$]  For $u\in\mathcal{S}'(\re^n)$, $u\in L^{2,k+tm(x,\x)}$ if and only if $\jap{x}^ku\in L^{2,tm(x,\x)}$. Moreover, there exists $C>0$ such that $u\in L^{2,k+tm(x,\x)}$
\begin{align*}
C^{-1}\|\jap{x}^ku\|_{L^{2,tm(x,\x)}}\leq \|u\|_{L^{2,k+tm(x,\x)}}\leq C\|\jap{x}^ku\|_{L^{2,tm(x,\x)}}.
\end{align*}
\end{itemize}
\end{lem}

\begin{proof}
$(i)$ This follows from the fact that the symbol of the inverse of $\Op(\tilde{G}_{k,tm})$ belongs to $S^{0,-k-tm(x,\x)}$.

$(ii)$ Note that $\Op(\tilde{G}_{0,tm})\jap{x}^k\Op(\tilde{G}_{k,tm})^{-1}$ and $\Op(\tilde{G}_{k,tm})\jap{x}^{-k}\Op(\tilde{G}_{0,tm})^{-1}$ is bounded in $L^2$ by Lemma \ref{varinv}. We are done.

\end{proof}

\end{document}